\documentclass[a4paper,11pt]{amsart}

\usepackage{latexsym,amssymb,amsfonts,amsmath}
\usepackage{amssymb,latexsym}
\usepackage{amsfonts}
\usepackage{color}
\usepackage{amsmath,amsthm,graphicx}
\usepackage[all]{xy}
\usepackage{hyperref}
\usepackage{a4wide}
\usepackage{ulem} 

\newtheorem{thm}{Theorem}[section]

\newtheorem{prop}[thm]{Proposition}

\theoremstyle{definition}
\newtheorem{defin}[thm]{Definition}
\newtheorem{example}[thm]{Example}
\theoremstyle{remark}
\newtheorem{remark}[thm]{Remark}
\numberwithin{equation}{section}

\begin{document}

\title[RFD and Kac quotients of orthogonal quantum groups]{The RFD and Kac quotients of the Hopf$^*$-algebras of universal orthogonal quantum groups}

\author{Biswarup Das}
\address{Instytut Matematyczny, Uniwersytet Wroc\l awski, pl.Grunwaldzki 2/4, 50-384 Wroc\l aw, Poland}
\email{biswarup.das@math.uni.wroc.pl}

\author{Uwe Franz}
\address{Laboratoire de math\'ematiques de Besan\c{c}on,
Universit\'e de Bourgogne Franche-Comt\'e, 16, route de Gray, 25 030
Besan\c{c}on cedex, France}
\email{uwe.franz@univ-fcomte.fr}
\urladdr{http://lmb.univ-fcomte.fr/uwe-franz}

\author{Adam Skalski}
\address{Institute of Mathematics of the Polish Academy of Sciences,
ul.~\'Sniadeckich 8, 00--656 Warszawa, Poland
}
\email{a.skalski@impan.pl}


\maketitle


\begin{abstract}
We determine the Kac quotient and the RFD (residually finite dimensional) quotient for the Hopf$^*$-algebras associated to universal orthogonal quantum groups.
\end{abstract}

\section{Introduction}\label{sec-intro}
Compact quantum groups of Woronowicz \cite{woronowicz87} are often studied via their associated Hopf$^*$-algebras, the so-called CQG algebras \cite{Koor}. The CQG algebra carries all the group-theoretic information about the associated quantum group, such as its representation theory, the lattice of quantum subgroups (described via the lattice of the CQG quotients of the original algebra), or Kac property, but also for example encodes approximation properties of the natural operator algebraic completions. 

When studying a particular property describing a `simpler' class of objects, it is natural to ask whether a general object admits a largest subobject with the given property. And thus So\l tan, motivated by the considerations concerning quantum group compactifications, showed in \cite{soltan05} (see also \cite{tomatsu07}) that every compact quantum group admits a unique maximal subgroup of Kac type; in other words, every CQG algebra admits a maximal Kac type quotient. He also computed such Kac quotients in some explicit examples, including the universal unitary quantum groups $U_Q^+$ of Wang and Van Daele. The same paper also saw the first seeds of the study of residually finite dimensional CQG algebras, fully developed ten years later by Chirvasitu \cite{chirvasitu15}. The latter article shows that every CQG algebra admits the RFD quotient, which roughly speaking is the largest quotient which has `sufficiently many' finite dimensional representations, discusses various stability results for the RFD property and most importantly proves that the CQG algebras of free unitary and orthogonal quantum groups, $U_n^+$ and $O_n^+$ are RFD for all $n\neq 3$. The case of $n=3$ was established later in \cite{chirvasitu20}. 
One should note that already combining \cite{soltan05}, \cite{chirvasitu15} and \cite{chirvasitu20} leads to the description of the RFD quotient of the CQG algebras of all $U_Q^+$. We also refer to these papers and their introduction for further motivation behind studying these concepts.

In this short note we compute the Kac and RFD quotients for the Hopf$^*$-algebras associated to universal orthogonal quantum groups $O_F^+$ of Wang and Van Daele, exploiting earlier results of Chirvasitu, the classification of $O_F^+$ up to isomorphism essentially due to Banica and Wang (formulated explicitly in \cite{derijdt07}), and the direct computations using the defining commutation relations. The main results are Theorems \ref{thm-RFD-O-CaseI} and 
 \ref{thm-RFD-O-CaseII}.

\section{Preliminaries}\label{sec-prelim}
We begin by recalling the basic objects and notions studied in this paper.
\subsection{Universal compact quantum groups}

We will study compact quantum groups in the sense of \cite{woronowicz87} via the associated CQG (compact quantum group) algebras. These are involutive Hopf algebras which are spanned by the coefficients of their finite-dimensional unitary corepresentations, see, e.g., \cite[Section 11.3]{ks97}; each of them admits a unique bi-invariant state, called the Haar state. Note that Hopf$^*$-quotients of CQG algebras are again CQG algebras, and the category of CQG algebras admits a natural free product construction (see for example \cite{Wangfree}).

The universal compact quantum groups $U_Q^+$ and $O_F^+$ were introduced by Van Daele and Wang \cite{wang+vandaele96}. Let $N \in \mathbb{N}$, let $F\in M_N(\mathbb{C})$ be invertible, and put $Q=F^*F$. The universal unitary CQG algebra ${\rm Pol}(U_{Q}^+)$, also denoted $A_u(Q)$, is generated by the $N^2$ coefficients of its fundamental corepresentation $U=(u_{jk})_{1\le j,k\le N}$, subject to the conditions that $U$ and $FUF^{-1}$ are unitaries in $M_N\big({\rm Pol}(U_Q^+)\big)$. This means that for all $1\le j,k\le N$ we have 
\begin{gather*} 
\sum_{\ell=1}^N u_{j\ell}u^*_{k\ell} = \delta_{jk} 1 = \sum_{\ell=1}^N u^*_{\ell j} u_{\ell k}, \tag{U1}\\
\sum_{\ell,r,s=1}^N u_{\ell j} (F^*F)_{\ell r}u^*_{rs} (F^*F)^{-1}_{sk} = \delta_{jk} 1 = \sum_{\ell,r,s=1}^N (F^*F)_{jr}u^*_{rs}(F^*F)^{-1}_{r\ell} u_{k\ell}. \tag{U2}
\end{gather*}
Thus the CQG algebra ${\rm Pol}(U_{F^*F}^+)$ depends only on the positive invertible matrix $Q$, which, up to isomorphism, we can assume to be diagonal, $Q=(\delta_{jk} q_j)_{1\le j,k\le N}$, with $0<q_1\le q_2\le \cdots\le q_N$.

If $F$ satisfies furthermore $F\overline{F}\in\mathbb{R}I_N$, then we define the universal orthogonal CQG algebra ${\rm Pol}(O_F^+)$, also denoted by $B_u(F)$ or $A_o(F)$, as the quotient of ${\rm Pol}(U_Q^+)$ by the additional relation
\[
U=F\overline{U}F^{-1}. \tag{H}
\]
Up to isomorphism of CQG algebras, it is sufficient to consider the following two families, see \cite{wang02} and \cite[Remark 1.5.2]{derijdt07}.

\begin{description}
\item[Case I] 
$F\overline{F}=I_N$, and $F$ can be written as
\begin{equation}\label{eq-F-caseI}
F=\left(
\begin{array}{ccc}
0 & D & 0 \\
D^{-1} & 0 & 0 \\
0 & 0 & I_{N-2k}
\end{array}
\right)
\end{equation}
with
\[
D=\left(
\begin{array}{ccc}
q_1 & & \\
& \ddots & \\
& & q_k
\end{array}
\right)
\]
a diagonal matrix with coefficients $0<q_1\le q_2 \le \cdots \le q_k<1$.

\item[Case II]
$F\overline{F}=-I_N$, $N$ is even, and $F$ can be written as
\begin{equation}\label{eq-F-caseII}
F=\left(
\begin{array}{cc}
0 & D  \\
-D^{-1} & 0 
\end{array}
\right)
\end{equation}
with 
\[
D=\left(
\begin{array}{ccc}
q_1 & & \\
& \ddots & \\
& & q_{N/2}
\end{array}
\right)
\]
a diagonal matrix with coefficients $0<q_1\le q_2 \le \cdots \le q_{N/2}\le 1$.
\end{description}
Note that the eigenvalues of $Q=F^*F$ are given by
\begin{gather*} 
\text{case I:} \quad 0< q^2_1 \le \cdots \le q^2_k < 1 < q_k^{-2} \le \cdots \le q_1^{-2}, \\
\text{case II:} \quad 0< q^2_1 \le \cdots \le q^2_{N/2} 
\le 1 \le q_{N/2}^{-2} \le \cdots \le q_1^{-2},
\end{gather*}
where in case I, $1$ is an eigenvalue only if $2k<N$.

\subsection{Kac quotient and RFD quotient}

If $A={\rm Pol}(\mathbb{G})$ is the CQG algebra of some compact quantum group, then the Kac ideal of $A$ is defined as the intersection of the (left) null spaces of all tracial states on $A$:
\[
\mathcal{J}_{\rm KAC} = \{ a\in A; \tau(a^*a)=0 \text{ for all tracial states }\tau\text{ on }A\}, 
\]
and the Kac quotient is $A_{\rm KAC}=A/\mathcal{J}_{\rm KAC}$. One can show that $A_{\rm KAC}$ is again a CQG algebra, which corresponds to the largest quantum subgroup of $\mathbb{G}$ which is of Kac type; the last statement means that the associated Haar state is a trace.

So{\l}tan \cite[Appendix A]{soltan05} \cite[Section 5]{soltan06} worked with the Kac quotient for C$^*$-algebras associated with compact quantum groups, but here we prefer to use a version for CQG algebras, which is also the setting in \cite{chirvasitu15}. 
See Subsection \ref{subsec-CQGvsCStar} below for a brief discussion of the relation between CQG-algebraic and C$^*$-algebraic Kac or RFD quotients.

Motivated by a question about Bohr compactifications of discrete quantum groups, Chirvasitu  introduced in \cite{chirvasitu15} the RFD property (where RFD stands for `residually finite dimensional') for CQG algebras and showed that ${\rm Pol}(U_N^+)=A_u(I_N)$ and ${\rm Pol}(O_N^+)=B_u(I_N)=A_o(I_N)$ have this property, implying that the discrete quantum groups $\widehat{U_N^+}$ and $\widehat{O_N^+}$ are maximal almost periodic in the sense of \cite{soltan05,soltan06}. See also the related more recent paper \cite{bbcw}.

The RFD quotient is defined as the biggest quotient of a CQG algebra that has the RFD property. We recall the relevant definitions from \cite{chirvasitu15}.

\begin{defin}\cite[Definition 2.6]{chirvasitu15}
A *-algebra $A$ has property \emph{RFD}, if for any $a\in A$, $a\not=0$,  there exists a finite-dimensional representation (i.e.\ a unital $^*$-homomorphism) $\pi:A\to M_n(\mathbb{C})$ with $\pi(a)\not=0$.

The \emph{RFD quotient} $A_{\rm RFD}$ of a *-algebra $A$ is the quotient of $A$ by the intersection of the kernels of all representations $\pi:A\to M_n(\mathbb{C})$, with $n \in \mathbb{N}$.
\end{defin}

In other words,  $A_{\rm RFD}=A/\mathcal{J}_{\rm RFD}$ with
\[
\mathcal{J}_{\rm RFD}=\{a\in A; \forall \pi:A\to M_n(\mathbb{C}) \text{ a  representation}, \pi(a)=0\}.
\]

One can show that the RFD quotient of a CQG algebra is again a CQG algebra.

Note that RFD is a weaker property than inner linearity (defined in \cite{bb2010}, see also \cite{bfs2012}); in general the relationship between various possible notions of residual finiteness for quantum groups remains not fully clarified -- see for example the comments in \cite{bbcw}.

Chirvasitu proved the following three results.

\begin{prop}\label{prop1}\cite[Last sentence of Section 2.4]{chirvasitu15}
If a CQG algebra has property RFD, then it is of Kac type.
\end{prop}

\begin{prop}\label{prop-free-prod}\cite[Proposition 2.10]{chirvasitu15}
If two *-algebras $A$ and $B$ have property RFD, then their free product $A\star B$ also has property RFD.
\end{prop}

\begin{thm}\label{thm-chir}\cite[Theorem 3.1]{chirvasitu15}, \cite[Theorem 2.4]{chirvasitu20}
The CQG algebras ${\rm Pol}(U_N^+)$ and ${\rm Pol}(O_N^+)$ have property RFD for $N\ge 2$.
\end{thm}

\begin{remark} \label{remark:com}
For $N=1$ we have ${\rm Pol}(U_1^+)=\mathbb{C}\mathbb{Z}$ and ${\rm Pol}(O_1^+)=\mathbb{C}\mathbb{Z}_2$, so property RFD also holds for $N=1$, cf.\ \cite[Remark 3.2]{chirvasitu15}. (More generally any commutative *-algebra that embeds into some C$^*$-algebra has RFD, cf.\ \cite[Remark 2.7]{chirvasitu15}).

The proofs in \cite{chirvasitu15} do not include $N=3$; this case is dealt with in \cite{chirvasitu20}.
\end{remark}

The quotient CQG-algebras $A_{\rm RFD}$ and $A_{\rm KAC}$ yield quantum subgroups $\mathbb{G}_{\rm RFD}$ and $\mathbb{G}_{\rm KAC}$ of $\mathbb{G}$. 
Since $\mathcal{J}_{\rm KAC}\subseteq \mathcal{J}_{\rm RFD}$, we have $\mathbb{G}_{\rm RFD}\subseteq \mathbb{G}_{\rm KAC}$, i.e.\  $A_{\rm RFD}$ is a quotient of $A_{\rm KAC}$.

\subsection{RFD quotient of universal unitary quantum groups}

The Kac quotients and the RFD quotients for the universal unitary quantum groups are already known, although the latter result has not been explicitly stated in the literature.
\begin{thm}\label{thm-RFD-U}
\cite{soltan05,chirvasitu15,chirvasitu20}
Let $Q\in M_d(\mathbb{C})$ be an invertible positive matrix with $r$ distinct eigenvalues $q_1,\ldots,q_r$, which have multiplicities $M_1,\ldots,M_r$. 

Then the Kac quotient and the RFD quotient of the CQG algebra ${\rm Pol}(U_Q^+)$ are equal to the free product
$\bigstar_{\nu=1}^r {\rm Pol}(U_{M_\nu}^+)$.
\end{thm}

\begin{remark}
So{\l}tan showed that this is the Kac quotient, cf\ \cite[Theorem 4.9]{soltan05} and \cite[Section 7]{soltan06}.
Chirvasitu's results, i.e., Proposition \ref{prop-free-prod} and Theorem \ref{thm-chir}, show that this free product is RFD, and therefore it is also the RFD quotient.
\end{remark}

\subsection{CQG-algebraic quotients vs.\ C$^*$-algebraic quotients}\label{subsec-CQGvsCStar}

Let $\mathbb{G}=(\mathsf{A},\Delta)$ be a compact quantum group with C$^*$-algebra $\mathsf{A}$ and CQG algebra $\mathcal{A}$. The C$^*$-algebraic Kac ideal and RFD ideal are
\[
\mathsf{J}_{\rm KAC} = \{ a\in \mathsf{A}; \tau(a^*a)=0 \text{ for all tracial states }\tau\text{ on }\mathsf{A}\},
\]
with $\mathsf{J}_{\rm KAC}=\mathsf{A}$ if $\mathsf{A}$ has no tracial states, and \[
\mathsf{J}_{\rm RFD} = \{ a\in \mathsf{A}; \pi(a)=0 \text{ for all fin.-dim.\ repr.\ }\pi\text{ of }\mathsf{A}\},
\]
with $\mathsf{J}_{\rm RFD}=\mathsf{A}$ if $\mathsf{A}$ has no finite-dimensional representations.

Again we can define $\mathsf{A}_{\rm KAC}$ and $\mathsf{A}_{\rm RFD}$ as respective quotients of $\mathsf{A}$ by $\mathsf{J}_{\rm KAC}$ and $\mathsf{J}_{\rm RFD}$, and again the RFD quotient is a quotient of the Kac quotient.

Since we can restrict tracial states or finite-dimensional representations of $\mathsf{A}$ to $\mathcal{A}$, we have
\[
\mathcal{J}_{\rm KAC} \subseteq \mathsf{J}_{\rm KAC}\cap\mathcal{A}
\quad\text{ and }\quad
\mathcal{J}_{\rm RFD} \subseteq \mathsf{J}_{\rm RFD}\cap\mathcal{A}.
\]
In general this inclusion can be proper. 
If $\mathsf{A}=C_u(\mathbb{G})$ is the universal C$^*$-algebra of $\mathbb{G}$, then we have equality, since every state and representation on $\mathcal{A}$ extends to $C_u(\mathbb{G})$.

\begin{example}\label{exa-cstar}
\cite[Proposition 2.4]{cs2019} showed that a compact quantum group is coamenable if and only its reduced C$^*$-algebra admits a finite-dimensional representation. Therefore, using the results of Banica from \cite{banica96}, \cite{banica97} and \cite{banica99} 
 we have $C_r(U_Q^+)_{\rm RFD}=\{0\}$ for $N\ge 2$, and $C_r(O_F^+)_{\rm RFD}=\{0\}$ for $N\ge 3$.

Banica \cite[Theorem 3]{banica97} showed also that the reduced C$^*$ algebra of $U_Q^+$ admits a unique trace if $Q\in\mathbb{R}I_N$, and no trace if this is not the case.
Thus we get
$C_r(U_N^+)_{\rm KAC}=C_r(U_N^+)$ and $C_r(U_Q^+)_{\rm KAC}=\{0\}$ if $Q\not\in \mathbb{R}I$ for $N\ge 2$. Similarly, if $N\ge 3$ and $\|F\|^8\le \frac{3}{4}{\rm Tr}FF^*$, then $C_r(O_F^+)$ has a unique trace if $F^*F=I$, and no trace if $F^*F\not\in \mathbb{R}I_N$, see \cite[Theorem 7.2]{vv07}. Therefore we get in this case $C_r(O_N^+)_{\rm KAC}=C_r(O_N^+)$ and $C_r(O_F^+)_{\rm KAC}=\{0\}$ if $F^*F\not\in \mathbb{R}I_N$.
\end{example}

\section{RFD quotient of the universal orthogonal quantum groups}\label{sec-RFD}

Let us now describe the RFD quotients of the free orthogonal quantum groups $O_F^+$ introduced in the beginning of the last section.

\subsection{Two special cases}\label{subsec-caseI}

Let us start with some special cases which will be useful in the next section when we treat the general situation.

\begin{prop}\label{prop-RFD-J}
Let $M\ge 1$ and let $J_M$ be the standard symplectic matrix
\[
J_M =
\left(
\begin{array}{cc}
0 & I_M \\
- I_M & 0
\end{array}
\right)
\]
Then the CQG algebra ${\rm Pol}(O_{J_M}^+)$ has property RFD.
\end{prop}
\begin{proof}
For $M=1$, we have $O_{J_1}^+=SU(2)$ and the result is true (as the algebra in question is commutative, see Remark \ref{remark:com}).

For the general case we can use the same proof as in \cite[Section 3]{chirvasitu15}.

\noindent{}
\emph{Step 1:}
The natural analog of \cite[Proposition 3.3]{chirvasitu15} holds.
Denote by $A'$ the unital *-subalgebra of $A={\rm Pol}(U^+_{2M})$ generated by $u_{jk}^* u_{\ell m}$, $1\le j,k,\ell,m\le 2M$ and by $B'$ the unital *-subalgebra of $B={\rm Pol}(O_{J_M}^+)$ generated by $u^*_{jk} u_{\ell m}$, $1\le j,k,\ell,m\le 2M$. Then there exists a unique CQG algebra isomorphism $A'\cong B'$ such that $A'\ni u^*_{jk}u_{\ell m}\mapsto u^*_{jk} u_{\ell m}\in B'$.

This isomorphism is simply the restriction to $A'$ of the embedding of ${\rm Pol}(U_{2M}^+)$ into $\mathbb{C}\mathbb{Z}\bigstar{\rm Pol}(O_{J_M}^+)$ defined in \cite[Th\'eor\`eme 1 (iv)]{banica97} by
\[
u_{jk} \mapsto z u_{jk}, \;\;\; j, k =1, \ldots, 2M,
\]
where $z$ denotes the generator of $\mathbb{Z}$ viewed as an element of $\mathbb{C}\mathbb{Z}$.

\noindent{}
\emph{Step 2:}
The center of ${\rm Pol}(O_{J_M}^+)$ is given by the morphism of CQG algebra $\gamma:{\rm Pol}(O_{J_M}^+)\to\mathbb{C}\mathbb{Z}_2$ with $\gamma(u_{jk})= \delta_{jk} t$ (where $t$ denotes the generator of $\mathbb{Z}_2$). The cocenter (i.e.\ the Hopf kernel of $\gamma$,  see \cite[Definition 2.10]{chirvasitu14})
is exactly $B'$. Indeed, $\gamma$ is central, i.e., it satisfies
\[
(\gamma \otimes {\rm id}) \Delta = (\gamma\otimes {\rm id})\circ \Sigma \circ \Delta: {\rm Pol}(O_{J_M}^+) \to  \mathbb{C}\mathbb{Z}_2 \otimes {\rm Pol}(O_{J_M}^+)
\]
where $\Sigma$ denotes the flip, and any other central map can be factored through $\gamma$. Furthermore, we have
\[
B' = {\rm Hker }(\gamma) = \{b\in {\rm Pol}(O_{J_M}^+): (\gamma \otimes {\rm id}) \Delta(b)=1\otimes b\}.
\]
\noindent{}
\emph{Step 3:}
We can therefore apply \cite[Theorem 3.6]{chirvasitu15} to prove an analogue of \cite[Proposition 3.8]{chirvasitu15}: ${\rm Pol}(O_{J_M}^+)$ is RFD if and only if ${\rm Pol}(U^+_{2M})$ is, and deduce from Theorem \ref{thm-RFD-U} above that ${\rm Pol}(O_{J_M}^+)$  indeed has property RFD.
\end{proof}

Let us consider next the case where $F^*F$ has only two eigenvalues: $q^2<1< q^{-2}$.

\begin{prop}\label{prop-one-block}
Let $M\in\mathbb{N}$, $q\in (0,1)$, $\epsilon\in\{-1,1\}$ and set
\[
F=
\left(
\begin{array}{cc}
0 & q I_M \\
\epsilon q^{-1} I_M & 0 
\end{array}
\right)
\]
Then the RFD quotient and the Kac quotient of the CQG algebra ${\rm Pol}(O_F^+)$ are both equal to the CQG algebra ${\rm Pol}(U_M^+)$. 
\end{prop}

\begin{proof}
This proof is similar to those of Theorems \ref{thm-RFD-O-CaseI} and \ref{thm-RFD-O-CaseII} in the next subsection, therefore we will give a rather detailed argument here, and later sketch only the main steps.
We decompose the fundamental corepresentation $U$ as
\[
U = \left(
\begin{array}{cc}
A & B \\
C & D
\end{array}
\right),
\]
with 
\begin{gather*}
A=(u_{jk})_{1\le j,k\le M}, B=(u_{jk})_{\genfrac{}{}{0pt}{2}{1\le j\le M}{M+1\le k\le 2M}} , C=(u_{jk})_{\genfrac{}{}{0pt}{2}{M+1\le j\le 2M}{1\le k\le M}}, \\
D= (u_{jk})_{M+1\le j,k\le 2M} \in M_M\big({\rm Pol}\big(O_F^+)\big).
\end{gather*}
The defining relation (H) of ${\rm Pol}(O_F^+)$ means that
\[
U = F \overline{U} F^{-1} =
\left(
\begin{array}{cc}
\overline{D} &\epsilon q^2 \overline{C}  \\
\epsilon q^{-2} \overline{B} & \overline{A} 
\end{array}
\right).
\]
So we can write $U$ as
\[
U = \left(
\begin{array}{cc}
A & \epsilon q^2 \overline{C}  \\
C & \overline{A}
\end{array}
\right),
\]
and therefore
\[
U^* = \left( 
\begin{array}{cc}
A^* &  C^* \\
\epsilon q^2 C^t     & A^t
\end{array}
\right).
\]
The unitarity condition for $U$ now reads
\begin{gather}\label{eq-unitary}
\left(
\begin{array}{cc}
A A^* + q^4 \overline{C} C^t & AC^* + \epsilon q^2 \overline{C} A^t \\
C A^* + \epsilon q^2 \overline{A} C^t & C C^* + \overline{A} A^t 
\end{array}
\right)
=
\left(
\begin{array}{cc}
I_M & 0 \\
0 & I_M
\end{array}
\right) =
\\
=
\left(
\begin{array}{cc}
A^* A+C^*C & \epsilon q^2 A^* \overline{C} + C^* \overline{A} \\
\epsilon q^2 C^t A + A^t C & q^4 C^t \overline{C} + A^t \overline{A}
\end{array}
\right).
\nonumber 
\end{gather}
The equalities of upper left corners of \eqref{eq-unitary} mean that for all $j,k=1, \ldots, M$
\[
\sum_{\ell=1}^M (u_{j\ell} u^*_{k\ell} + q^4 u^*_{j+M,\ell}u_{k+M,\ell}) = \delta_{jk} 1 = \sum_{\ell=1}^M  (u^*_{\ell j}u_{\ell k} + u^*_{\ell +M,j}u_{\ell+M,k} ).
\]
Setting $j=k$ and taking the sum, we get
\[
\sum_{j,\ell=1}^M (u_{j\ell} u^*_{j\ell} - u^*_{j\ell} u_{j\ell})
=
\sum_{j,\ell=1}^M (1-q^4) u^*_{j+M,\ell} u_{j+M,\ell}.
\]
Let $\tau$ be a tracial state on ${\rm Pol}(O_F^+)$. The equality above implies that
\[
\tau(u_{j+M,\ell}^*u_{j+M,\ell}) = 0
\]
for all $j,\ell\in\{1,\ldots,M\}$. So the generators $u_{jk}$ with $M+1\le j\le 2M$ and $1\le k\le M$, which form the matrix $C$, belong to the Kac ideal $\mathcal{J}_{\rm KAC}$.

If we divide by the *-ideal generated by the coefficients of $C$, then we see from Equation \eqref{eq-unitary} that the remaining generators $u_{jk}$ with $1\le j,k\le M$, which form the matrix $A$ --- or rather their images in the quotient *-algebra --- have to satisfy exactly the defining relations of ${\rm Pol}(U_M^+)$, i.e.,
\[
AA^*= I_M = A^*A \quad\text{ and }\quad \overline{A}A^t = I_M = A^t \overline{A}.
\]
The result now follows, since Chirvasitu proved that ${\rm Pol}(U_M^+)$ is RFD, cf.\ Theorem \ref{thm-chir}.
\end{proof}

\subsection{Case I: $F\overline{F}=I_N$}\label{subsec-caseII}

We now look at the case $F\overline{F}=I_N$, where we can assume that $F$ has the form given in Equation \eqref{eq-F-caseI}. But we will permute the rows and columns of $F$ to organize $F$ in blocks corresponding to the eigenvalues of $F^*F$.

\begin{thm}\label{thm-RFD-O-CaseI}
Let $F$ be of the form
\[
F = 
\left(
\begin{array}{cccccc}
0 & q_1 I_{M_1} & & & & \\ 
q_1^{-1} I_{M_1} & 0 & & & & \\
&& \ddots &&& \\
&&& 0 & q_r I_{M_r} & \\
&&& q_r^{-1} I_{M_r} & 0 & \\
&&&&& I_{N-2K}
\end{array}
\right)
\]
with $0<q_1<\cdots <q_r < 1$ and $K=M_1+\cdots+M_r$.

Then the RFD quotient and the Kac quotient of the CQG algebra ${\rm Pol}(O_F^+)$ are both equal to the free product
\[
\left(\bigstar_{\nu=1}^r {\rm Pol}(U_{M_\nu}^+)\right) \bigstar {\rm Pol}(O^+_{N-2K}).
\]
\end{thm}
\begin{proof}
The proof is similar to that of Proposition \ref{prop-one-block}.

Writing $U$ as a block matrix and using the relation between the blocks that follow from (H), we can express $U$ as
\begin{equation}\label{eq-formU}
U=\left( 
\begin{array}{cccccccc}
A_{11} & q_1^2 \overline{C_{11}} & A_{12} & q_1q_2\overline{C_{12}} & \hdots & R_1 \\[5pt]
C_{11} & \overline{A_{11}} & C_{12} & q_1q_2^{-1}\overline{A_{12}} & \hdots & q_1^{-1}\overline{R_1} \\[5pt]
A_{21} & q_2q_1\overline{C_{21}} & A_{22} & q_2^2 \overline{C_{22}} & \hdots & R_2 \\[5pt]
C_{21} & q_2q_1^{-1}\overline{A_{21}} & C_{22} & \overline{A_{22}} & \hdots & q_2^{-1}\overline{R_2} \\[5pt]
\vdots & \vdots & \vdots & \vdots &\ddots& \vdots \\[5pt]
X_1 & q_1\overline{X_1} & X_2 & q_2\overline{X_2} & \hdots & Z
\end{array}
\right),
\end{equation}
where furthermore the coefficients of $Z$ are hermitian, i.e., $Z=\overline{Z}$.

If we look at the diagonal blocks of the unitarity condition $U^*U=I_N=UU^*$, we get for every $\mu =1, \ldots,r$
\begin{gather}
\sum_{\rho=1}^r \left(A^*_{\rho\mu} A_{\rho\mu} + C^*_{\rho\mu}C_{\rho\mu}\right) + X_\mu^* X_\mu
= I_{M_\mu}
\nonumber
\\
= \sum_{\rho=1}^r \left(A_{\mu\rho}A^*_{\mu\rho} + q_\mu^2 q^2_\rho \overline{C_{\mu\rho}} C^t_{\mu\rho}\right) + R_\mu R^*_\mu, 
\label{eq-RFD-O-condI}
\\
\sum_{\rho=1}^r \left(q^2_\mu q^2_\rho C^t_{\rho\mu} \overline{C_{\rho\mu}} + q^2_\rho q^{-2}_\mu A^t_{\rho\mu}\overline{A_{\rho\mu}}\right) + q^2_\mu X^t_\mu\overline{X_\mu}  = I_{M_\mu}
\nonumber 
\\
= \sum_{\rho=1}^r \left(C_{\mu\rho}C^*_{\mu\rho} + q^2_\mu q^{-2}_\rho \overline{A_{\mu\rho}} A^t_{\mu\rho}\right) + q_\mu^{-2} \overline{R_\mu}R^t_\mu, 
\label{eq-RFD-O-condII}
\\
Z^tZ + \sum_{\rho=1}^r \left( R^*_\mu R_\mu + q_\mu^{-2} R_\rho^t \overline{R_\rho}\right)
= I_{N-2K} = ZZ^t + \sum_{\rho=1}^r \left(X_\rho X^*_\rho + q^2_\rho \overline{X_\rho} X^t_\rho\right).
\label{eq-RFD-O-condIII}
\end{gather}
Note that if
\[
A=(a_{jk})_{\genfrac{}{}{0pt}{2}{1\le j \le J}{1\le k\le K}}\in M_{J\times K}(A)
\]
is a matrix with coefficients in some *-algebra $A$ and $\tau$ is a tracial state on $A$, then we have
\[
\tau\circ{\rm Tr}(A^*A) = \sum_{j=1}^J\sum_{k=1}^K \tau(a^*_{jk} a_{jk}) = \tau\circ{\rm Tr}(AA^*) = \tau\circ{\rm Tr}(\overline{A}A^t) = \tau\circ{\rm Tr}(A^t\overline{A}).
\]
So if $\tau$ is a tracial state on ${\rm Pol}(O_F^+)$ and we apply $\tau\circ{\rm Tr}$ to Equation \eqref{eq-RFD-O-condIII}, then we get
\begin{equation}\label{eq-RFD-diff}
\sum_{\rho=1}^r (1+q_\rho^2) \tau\big({\rm Tr}(X^*_\rho X_\rho)\big) = \sum_{\rho=1}^r (1+q_\rho^{-2})\tau\big({\rm Tr}(R_\rho^*R_\rho)\big).
\end{equation}

If we now take the sum over $\mu$ of the difference between the left-hand-side and the right-hand-side in Equations \eqref{eq-RFD-O-condI} and \eqref{eq-RFD-O-condII}, and apply $\tau\circ{\rm Tr}$, then we get
\begin{gather*}
\sum_{\rho,\mu=1}^r (1-q_\mu^2q_\rho^2)\tau\big({\rm Tr}(C_{\rho\mu}^*C_{\rho\mu})\big) + \sum_{\mu=1}^r \tau\big({\rm Tr}(X^*_\mu X_\mu)\big) - \sum_{\mu=1}^r \tau\big({\rm Tr}(R_\mu^* R_\mu))\big) = 0, 
\\
\sum_{\rho,\mu=1}^r (1-q_\mu^2q_\rho^2)\tau\big({\rm Tr}(C_{\rho\mu}^*C_{\rho\mu})\big) + \sum_{\mu=1}^r q_\mu^2 \tau\big({\rm Tr}(X^*_\mu X_\mu)\big) - \sum_{\mu=1}^r q_{\mu}^{-2} \tau\big({\rm Tr}(R_\mu^* R_\mu)\big) = 0. 
\end{gather*}
Adding these two relations and taking Equation \eqref{eq-RFD-diff} into account, we get $\tau\big({\rm Tr}(C^*_{\rho\mu}C_{\rho\mu})\big)=0$ for all $\rho,\mu\in\{1,\ldots,r\}$; by positivity this means that all the generators that appear in the $C$-blocks are contained in the Kac ideal $\mathcal{J}_{\rm KAC}$.

By \eqref{eq-RFD-O-condI}, we then also have $\tau\big({\rm Tr}(X_\mu^*)X_\mu\big)=\tau\big({\rm Tr}(R_\mu^*R_\mu)\big)$, so, plugging this into \eqref{eq-RFD-diff}, 
\[
\sum_{\rho=1}^r (q^{-2}_\rho - q_\rho^2) \tau\big({\rm Tr}(X^*_\rho X_\rho)\big) = 0,
\]
and we get $\tau\big({\rm Tr}(X_\mu^*X_\mu)\big)=0=\tau\big({\rm Tr}(R_\mu^*R_\mu)\big)$ for $\mu=1,\ldots, r$, since all terms in the above sum are non-negative. Once again using the fact that $\tau$ is positive we deduce that all the generators that appear in the $X$- and $R$-blocks are contained in the Kac ideal $\mathcal{J}_{\rm KAC}$.

Denote by
\[
A=(A_{\rho\mu})_{1\le\rho,\mu\le r} \in M_K\big({\rm Pol}(O_F^+)\big)
\]
the matrix obtained from $U$ by deleting the generators in the even rows and columns in the block decomposition in Equation \eqref{eq-formU}, as well as the last row and column.

If we divide the *-algebra ${\rm Pol}(O_F^+)$ by the *-ideal generated by all $C_{\rho\mu}$, $X_\mu$ and $R_\mu$, then unitarity relation $U^*U=I_N=UU^*$ reduces to
\begin{gather*}
A^*A = I_K = AA^* \quad\text{ and }\quad D\overline{A}D^{-1}A^t = I_K = A^t D\overline{A}D^{-1}, \\
Z=\overline{Z} \quad\text{ and }\quad ZZ^t = I_{N-2K} = Z^t Z,
\end{gather*}
where
\[
D=\left(
\begin{array}{ccc}
q^2_1 I_{M_1} & &  \\
& \ddots & \\
&& q^2_r I_{M_r} 
\end{array}
\right).
\]
This means that the quotient ${\rm Pol}(O_F)/\langle C_{\rho\mu}, X_\mu, R_\mu:\rho,\mu=1,\ldots,r\rangle$ is equal to the free product of a copy of ${\rm Pol}(U_D^+)$, generated by the coefficients of the $A_{\rho\mu}$, $\rho,\mu=1,\ldots,r$, and a copy of ${\rm Pol}(O_{N-2K}^+)$, generated by the coefficients of $Z$.

Now we can conclude with Theorem \ref{thm-RFD-U}.
\end{proof}

\subsection{Case II: $F\overline{F}=-I_N$}

Let us now consider the case $F\overline{F}=-I_N$ and $F$ a matrix of the form given in Equation \eqref{eq-F-caseII}.

\begin{thm}\label{thm-RFD-O-CaseII}
Let $N$ be an even positive integer and let $F\in M_N$ be of the form
\[
F = 
\left(
\begin{array}{ccccc}
0 & q_1 I_{M_1} & & & \\
-q_1^{-1} I_{M_1} & 0 & & & \\
&& \ddots && \\
&&& 0 & q_r I_{M_r} \\
&&& -q_r^{-1} I_{M_r} & 0 \\
\end{array}
\right),
\]
with $0<q_1<\cdots q_{r-1}<q_r=1$, $M_1,\ldots,M_{r-1}\ge 1$, $M_r\ge 0$, $M_1+\cdots+M_r=N/2$. Note that $M_r=0$ if $1$ is not an eigenvalue of $F^*F$.

The RFD quotient and the Kac quotient of the CQG algebra ${\rm Pol}(O_F^+)$ are both equal to the free product
\[
\left(\bigstar_{\nu=1}^{r-1} {\rm Pol}(U_{M_\nu}^+)\right) \bigstar {\rm Pol}(O^+_{J}),
\]
with
\[
J =
\left(
\begin{array}{cc}
0 & I_{M_r} \\
-I_{M_r} & 0
\end{array}
\right)
\]
\end{thm}

\begin{proof}
Like in the proofs of Proposition \ref{prop-one-block} and Theorem \ref{thm-RFD-O-CaseI}, we write $U$ as a block matrix. Since $U=F\overline{U}F$, we can write
\[
U = \left(
\begin{array}{ccccccc}
A_{11} & -q_1^2 \overline{A_{11}} & A_{12} & -q_1q_2 \overline{C_{12}} & \hdots & A_{1r} & -q_1q_r \overline{C_{1r}} \\
C_{11} & \overline{A_{11}} & C_{12} & q_1^{-1}q_2\overline{A_{12}} & \hdots & C_{1r} & q_1^{-1}q_r \overline{A_{1r}} 
\\
A_{21} & -q_2q_1 \overline{C_{21}} & A_{22} & -q^2_2 \overline{C_{22}} & \hdots & A_{2r} & -q_2q_r \overline{C_{2r}} \\
C_{21} & q_2^{-1} q_1 \overline{A_{21}} & C_{22} & \overline{A_{22}} & \hdots & C_{2r} & q_2^{-1}q_r \overline{A_{2r}} 
\\
\vdots & \vdots & \vdots & \vdots &\ddots & \vdots & \vdots 
\\
A_{r1} & q_r q_1 \overline{C_{r1}} & A_{r2} & -q_2q_2 \overline{C_{r2}} &\hdots & A_{rr} & -q_r^2 \overline{C_{rr}} \\
C_{r1} & q_r^{-1}q_1 \overline{A_{r1}} & C_{r2} & q_r^{-1}q_2 \overline{A_{r2}} & \hdots & C_{rr} & \overline{A_{rr}}
\end{array}
\right).
\]
The unitarity conditions on the diagonal blocks read ($\nu =1,\ldots,r$)
\begin{gather}
\sum_{\mu=1}^r \left(A^*_{\mu\nu} A_{\mu\nu} + C^*_{\mu\nu} C_{\mu\nu} \right)
= I_{M_\nu}
=\sum_{\mu=1}^r \left(A_{\nu\mu} A^*_{\nu\mu} + q^2_\mu q^2_\nu \overline{C_{\mu\nu}} C^t_{\mu\nu} \right), 
\label{eq-RFD-caseII-condI}\\
\sum_{\mu=1}^r \left( q_\mu^2 q_\nu^2 C^t_{\mu\nu} \overline{C_{\mu\nu}} + q^{-2}_\mu q^2_\nu A^t_{\mu\nu} \overline{A_{\mu\nu}} \right)
= I_{M_\nu}
=\sum_{\mu=1}^r \left( C_{\nu\mu} C^*_{\nu\mu} + q_\nu^{-2}q_\mu^2 \overline{A_{\nu\mu}} A^t_{\nu\mu} \right).
\label{eq-RFD-caseII-condII}
\end{gather}

Letting $\tau$ be a tracial state on ${\rm Pol}(O_F^+)$ and applying $\tau\circ{\rm Tr}$ to the difference of the left-hand-side and the right-hand-side in Equation \eqref{eq-RFD-caseII-condI}, we get
\[
\sum_{\nu,\mu=1}^r (1-q_\mu^2 q_\nu^2) \tau\big({\rm Tr}(C_{\mu \nu}^*C_{\mu \nu})\big) =0,
\]
which implies that the coefficients appearing in all the $C$-blocks, except possibly $C_{rr}$, belong to the Kac ideal $\mathcal{J}_{\rm KAC}$.

Taking now the differences of the left-hand-sides, or, respectively, right-hand-sides, in Equations \eqref{eq-RFD-caseII-condI} and \eqref{eq-RFD-caseII-condII},  we get
\begin{gather*}
\sum_{\mu=1}^{r-1} (q_\mu^{-2} q_\nu^2-1) \tau\big({\rm Tr}(A_{\mu\nu}^*A_{\mu\nu})\big) =0, \\
\sum_{\mu=1}^{r-1} (1-q^2_\mu q_\nu^{-2}) \tau\big({\rm Tr}(A_{\nu\mu}A_{\nu\mu})^*\big) =0, 
\end{gather*}
for $\nu\in\{1,\ldots,r-1\}$. 
From these two relations we can prove by induction that $\tau\big({\rm Tr}(A_{\mu\nu}^*A_{\mu\nu})\big)=0$ for all $\mu,\nu=1, \ldots, r-1$ with $\mu\not=\nu$.

Denote by $\mathcal{J}$ the *-ideal generated by the $u_{jk}$ that have been regrouped in the blocks $C_{\mu\nu}$ with $(\mu,\nu)\not=(r,r)$, and in the blocks $A_{\mu\nu}$ with $\mu\not=\nu$.
It is not difficult to show that ${\rm Pol}(O_F^+)/\mathcal{J}\cong \left(\bigstar_{\nu=1}^r {\rm Pol}(U_{M_\nu}^+)\right) \bigstar {\rm Pol}(O^+_{F_0})$. As the latter CQG algebra is RFD by Chirvasitu's results, it follows that this is indeed the RFD quotient and also the Kac quotient of ${\rm Pol}(O_F^+)$.
\end{proof}

\section*{Acknowledgements}
 We acknowledge support by the French 
MAEDI and MENESR and by the Polish MNiSW through the Polonium programme.
AS was partially supported by the NCN (National Science Centre) grant 2014/14/E/ST1/00525.
UF was supported by the French `Investissements d’Avenir' program, project ISITE-BFC
(contract ANR-15-IDEX-03), and by an ANR project (No. ANR-19-CE40-0002).
We thank Anna Kula for several discussions on related topics.


\end{document}